\documentclass[10pt]{amsart}

\usepackage{amsmath}
\usepackage{amsthm}
\usepackage{amsfonts}
\usepackage{amssymb,latexsym}
\usepackage[all]{xy}
\usepackage{graphicx}
\usepackage[mathscr]{eucal}
\usepackage{verbatim}
\usepackage{hyperref}

\theoremstyle{plain}
    \newtheorem{thm}{Theorem}[section]
    \newtheorem{prop}[thm]{Proposition}
    \newtheorem{lemma}[thm]{Lemma}
    
    \newtheorem{cor}[thm]{Corollary}

\theoremstyle{definition}
    
    \newtheorem{rem}[thm]{Remark}

\theoremstyle{remark}
    \newtheorem{example}[thm]{Example}

\numberwithin{equation}{section}

\newcommand{\rar}{\ensuremath{\rightarrow}}
\newcommand{\lrar}{\ensuremath{\longrightarrow}}

\newcommand{\G}{\mathcal G}

\newcommand{\ftwo}{\mathbb{F}_2}
\newcommand{\fp}{\mathbb{F}_p}

\newcommand{\Circ}{\mathrm{circ}}
\newcommand{\ints}{\mathbb{Z}}
\newcommand{\ord}{\textup{ord}}
\def \a{\alpha}
\newcommand{\FF}{\mathbb{F}}
\newcommand{\map}[3]{\ensuremath{#1 : #2 \longrightarrow #3}}

\begin{document}

\title{Characterizations of  Mersenne and $2$-rooted primes}
\date{\today}

\author{Sunil K. Chebolu}
\address{Department of Mathematics \\
Illinois State University \\
Normal, IL 61790, USA} \email{schebol@ilstu.edu}
\urladdr{http://math.illinoisstate.edu/schebol/ }

\author{Keir Lockridge}
\address{Department of Mathematics \\
Gettysburg College\\
Gettysburg, PA 17325} \email{klockrid@gettysburg.edu}
\urladdr{http://keir.gettysburgmath.org}

\author{Gaywalee Yamskulna}
\address{Department of Mathematics \\
Illinois State University \\
Normal, IL 61790, USA} \email{gyamsku@ilstu.edu}

\thanks{
The first author is supported by an NSA grant (H98230-13-1-0238) and the third author from a Simons Foundation Collaboration Grant (207862)}


\begin{abstract}
We give several characterizations of Mersenne primes (Theorem \ref{main}) and of primes for which $2$ is a primitive root (Theorem \ref{main2}). These characterizations involve group algebras, circulant matrices, binomial coefficients, and bipartite graphs.
\end{abstract}

\maketitle
\thispagestyle{empty}

\tableofcontents

\section{Introduction}

A Mersenne prime is a prime number of the form $2^n-1$ for some positive integer $n$.  (It is easy to see that if $n$ is composite, then so is $2^n-1$. Therefore we may assume that $n$ is prime in the definition of a Mersenne prime.)   These primes were named after the French monk Marin Mersenne who studied them in the early 17th century, but they appeared much earlier. In the 4th century BC, Euclid showed that   if $2^p-1$ is prime for a prime $p$, then $2^{p-1}(2^p-1)$ is a perfect number.  (Recall that a positive integer  is perfect if it equals the sum of its proper divisors.)  Euler studied these primes in the 17th century when he proved the converse of Euclid's theorem.  We refer the reader to \cite{Burton} for a historical survey on Mersenne primes. In the late 1990s, the GIMPS (Great Internet Mersenne Prime Search) project rejuvenated interest in these primes.  Although there has been much theoretical and computational research on Mersenne primes, basic questions about them remain open. For instance, it is not known whether there are infinitely many such primes.

In this paper we obtain several characterizations of Mersenne primes (Theorem \ref{main}) and also of primes for which $2$ is a primitive root (Theorem \ref{main2}).  Some of these characterizations are obtained by studying groups of units in group algebras; others are based on  binomial coefficients, circulant matrices, and  bipartite graphs.  Our characterizations will therefore  translate theorems, unsolved problems, and conjectures  about these primes into  other areas of mathematics including commutative algebra, linear algebra and graph theory. 

Our results on Mersenne primes are summarized in the following theorem. Let $C_p$ denote the (multiplicative) cyclic group of order $p$.

\begin{thm}[Mersenne Primes] \label{main}
Let $p > 3$ be a prime. Then the following statements are equivalent.
\begin{enumerate}
\item  The prime $p$ is a Mersenne prime.
\item  There is an non-trivial abelian group $G$ and a field $k$ such that every non-trivial unit in $kG$ has order $p$.
\item  There is a non-trivial group $G$ and a field $k$ such that every non-trivial unit in $kG$ has order $p$.
\item  Every non-trivial unit in $\ftwo C_p$ has order $p$.
\item $(1+x+x^2)^p = 1$ in $\ftwo C_p$, where $x$ is a generator of $C_p$.
\item $(1+x)^p$ = $(1+x^3)^p$ in $\ftwo C_p$, where $x$ is a generator of $C_p$.
\item${p \choose r} \equiv {p \choose 3r \mod p} \mod 2$ for all $1 \le r \le p-1$.
\item The group of  $p \times p$ invertible circulant matrices over $\ftwo$ is an elementary abelian $p$-group.
\item The circulant matrix  $\Circ(1, 1, 1, 0, 0, \cdots, 0)$ is of order $p$ in the ring of $p \times p$ matrices over $\ftwo$. (See Section \ref{matrices} for  definitions.)
\item $[\Circ(1,1,0, 0, \cdots, 0)]^p = [\Circ(1,0,0,1, 0, \cdots, 0)]^p$  in the ring of  $p \times p$ matrices over $\ftwo$.
\item Every circulant  $(p ,p)$ bipartite graph with odd number of perfect matchings has
$s_{ij}(p) \mod 2 = \delta_{ij}$, where $s_{ij}(p)$ is the number of pseudopaths between vertex $a_i$ and vertex $b_j$ and $\delta_{ij}$ is the Kronecker delta symbol. (See Section \ref{graphs} for  definitions.)
\item The $(p,p)$ bipartite graph corresponding to the $p \times p$ circulant matrix  \break $\Circ(1, 1, 1, 0, 0, \cdots, 0)$ has
$s_{ij}(p) \mod 2 = \delta_{ij}$.
\end{enumerate}
\end{thm}
Note that when $p =3$,  which is a Mersenne prime, the element $1+x+x^2$ is a zero divisor in the ring $\ftwo C_3$ because $(1+x+x^2)(1+x) = 1 +x^3 = 1+1 = 0$.  Therefore the above theorem breaks down partially when $p=3$. However, it will be clear from our analysis that even in this case,  statements 1, 2, 3, 4, 8, and 11 are equivalent.

Several characterizations of Mersenne primes in connection to binomial coefficients, number-theoretic functions,  and units in a ring can be found in the literature. Here we state a few characterizations.
 An odd prime $p$ is Mersenne if and only if all the binomial coefficients ${p \choose r}$  ($0 \le r \le p$) are odd numbers; see \cite[Theorem 8.14]{Kosher}.  The $n$th Catalan number is defined to be the quantity
\[ \frac{1}{n+1} {2n \choose n}.\]
The $p$th Catalan number is odd if and only if $p$ is a Mersenne prime; see \cite[Theorem 8.15]{Kosher}.
 Let $\sigma(n)$ denote the sum of the positive divisors of a positive integer $n$.  The quantity $\sigma(n)$ is a power of $2$ if and only if $n$ is a product of distinct Mersenne primes; see \cite[Example 8.22]{Kosher}. There is a ring $R$ with exactly $p$ units if and only if $p$ is a Mersenne prime; see \cite{David}. Our Theorem \ref{main} adds several items to this list.
 
We now turn our attention to  primes for which $2$ is a primitive root (primes $p$ for which $2$ generates the multiplicative group of the field with $p$ elements).  It is well known that every odd prime $p$ has a primitive root. Since the multiplicative group of the field with $p$ elements is a cyclic group of order $p-1$, we know that in fact  there are $\phi(p-1)$ primitive roots mod $p$, where $\phi$ is Euler's totient function.  However, there is no known formula or even a polynomial-time algorithm for finding  a  primitive root. So one looks at the inverse primitive root problem. That is, we fix an integer $a$ and ask: for which odd primes $p$ will $a$ be a primitive root?  For $a$ to be a primitive root mod $p$, there are some obvious necessary conditions on $a$. For instance, $a$ cannot be $-1$, because $(-1)^2=1$. Similarly, it is easy to see that $a$ cannot be a perfect square because a primitive root has to be a quadratic non-residue mod $p$.  A deep conjecture of Artin says that these
two conditions on $a$ are sufficient to guarantee the existence of infinitely many primes $p$ for which $a$ will be a primitive root.
\vskip 2mm
\noindent
\textbf{Artin's Conjecture:}   Let $a$ be an integer which is not a perfect square and not equal to $-1$. Then $a$ is a primitive root mod $p$ for infinitely many primes $p$.
\vskip 2mm
There is no  single specific value of $a$ for which Artin's conjecture is resolved.  However, it is known that the Generalized Riemann Hypothesis implies Artin's Conjecture. We refer the reader to  \cite{Hooley, Gupta-Murthy} for more details. The smallest positive integer $a$ that satisfies the conditions of Artin's conjecture is $a = 2$.  The corresponding special case of Artin's Conjecture is the statement that there are infinitely many primes $p$ for which $2$ is a primitive root.  We will call such primes $2$-rooted.   In this paper we offer several characterizations of these primes, summarized in the next theorem.

\begin{thm}[2-rooted Primes] \label{main2}
Let $p$ be an odd prime. Then the following are equivalent.
\begin{enumerate}
\item The prime $2$ is a primitive root mod $p$.
\item$|(\ftwo C_p)^\times| = 2^{p-1}-1$.
\item The only units in $\ftwo C_p$ which have order $p$ are the non-identity elements of $C_{p}$.
\item If $\theta$ is an element of $\ftwo C_p$ which is not the norm element (the sum of all the group elements) and which is not in the kernel of the augmentation map, then $\theta$ is a unit.
\item There are $2^{p-1}-1$ invertible circulant matrices of size $p \times p$ over $\ftwo$.
\item The only invertible circulant matrices over $\ftwo$ that have order dividing $p$ are the circulant permutation matrices.  (See Section \ref{matrices} for definitions.)
\item If $A$ is a  $p \times p$ circulant matrix over $\ftwo$ which is not $J$ (the $p \times p$ matrix with all 1's) and the vector of all 1's is not in the null space of $A$, then $A$ is invertible over $\ftwo$. 
\item There are $2^{p-1}-1$  circulant $(p,p)$ bipartite graph on labeled vertices with an odd number of perfect matchings.  (See Section \ref{graphs} for  definitions.)
\item If $G$ is a circulant $(p,p)$ bipartite graphs on labeled vertices with odd degree and if $G$ not a complete bipartite graph, then $G$ has an odd number of perfect matchings.
\item If $G$ is a circulant $(p,p)$ bipartite graph on labeled vertices with an odd number of perfect matchings and $s_{ij}(p) \mod 2  = \delta_{ij}$ for all $1\le i , j \le p$, then the degree of $G$ is 1. (See Section \ref{graphs} for  definitions.)
\end{enumerate}
Moreover, if any of these statements hold, then $p \equiv 3 \text{ or }  5 \mod 8.$
\end{thm}

There is another interesting characterization  of these primes which occurs in connection to the Josephus problem \cite{transitive}. Let $p$ be a odd prime expressed as $2m+1$ for some positive integer $m$.  Then the permutation 
\[(1,2) (1,2,3) \cdots (1,2, \cdots,m)\] 
is transitive (which simply means that it is a single cycle containing all of $1,2, \cdots ,m$) if and only if $p$ is $2$-rooted; see \cite{transitive}.

By combining Theorem \ref{main} and Theorem \ref{main2} we recover the
following result in number theory.

\begin{cor}
The prime $3$ is the only prime which is both Mersenne and $2$-rooted.
\end{cor}

\noindent This result is often obtained as an immediate consequence of the quadratic reciprocity law. Our proof is different and is relatively more elementary than the one which uses the reciprocity law; see Proposition \ref{3}.

We now explain how we arrived at these characterizations. It is interesting to note that the original problem  which  led to these characterizations  seemed  to have nothing to do with Mersenne primes or $2$-rooted primes.  To explain further, we need  a definition.
A ring $R$ is said to have the  \emph{diagonal property} if its multiplication table has 1's only on the diagonal.  More precisely, this means that whenever $ab=1$ in $R$, then $a = b$.   The diagonal property for rings was introduced by the first author in \cite{24}, where it was shown that $\mathbb{Z}_n$ has the diagonal property if and only if $n$  is a divisor of 24. (In \cite{24}, the reader will find  5 different proofs of this fundamental result.
These proofs are based on the Chinese Remainder Theorem, Dirichlet's theorem on primes in an arithmetic progression, the structure of units in $\mathbb{Z}_n$, the Bertrand-Chebyshev Theorem, and generalizations of the Bertrand-Chebyshev Theorem by Erd\"{o}s and Ramanujan.)
 In \cite{12}, the first author and Mayers proved that the diagonal property holds for the ring of polynomials in $m$ commuting variables over $\mathbb{Z}_n$ if and only if $n$ is a divisor of 12.  (Note that the answer is independent of $m$.)  In \cite{GL}, the second author and Genzlinger consider the proportion of units of order at most 2 in $\mathbb{Z}_n$, proving for example that this proportion is the reciprocal of a prime $p$ if and only if $p$ is a Sophie Germain prime. Continuing this line of research, in Section 2 we investigate the diagonal property for group algebras and prove the following result. Write $C_p^r$ to denote the direct sum of $r$ copies of $C_p$.

\begin{thm}\label{main-2-kg}
Let $G$ be a group and let $k$ be a field. The group algebra $kG$ has the diagonal property if and only if $kG$ is either
$\mathbb{F}_2 C_2^r $ or  $ \mathbb{F}_3 C_2^r $ for some $0 \leq r \le \infty$.
\end{thm}

After proving this result, we consider the following natural generalization of the diagonal property. Call a ring $R$ a {\em $\Delta_n$-ring} if $u^n = 1$ for every unit of $R$.  If $n$ is the smallest positive integer such that $R$ satisfies this property, then call $R$ a {\em strict $\Delta_n$-ring}.  A ring $R$ satisfies the diagonal property if and only if $R$ is a $\Delta_2$-ring. This generalization is what led us to Mersenne primes. In Section \ref{section-kg}, we prove the following theorem.

\begin{thm}\label{main-p-kg}
Let $G$ be a non-trivial abelian group, let $k$ be a field, and let $p$ be an odd prime. The group algebra $kG$ is a $\Delta_p$-ring if and only if $p$ is a Mersenne prime and
$kG$ is either  $\ftwo (C_p^r)$ or  $\mathbb{F}_{p+1} (C_p^r)$ for some $0 < r \leq \infty$.
\end{thm}

\noindent It was the proof of this theorem and related investigations that gave us several of the characterizations of Mersenne primes and 2-rooted primes mentioned in Theorems \ref{main} and \ref{main2}.

\vskip 3mm
\noindent \textbf{Organization.} In Section \ref{section-kg}, we find all group algebras which are $\Delta_2$-rings and all abelian group algebras which are $\Delta_p$-rings for $p$ an odd prime. In Section \ref{units}, we examine the structure of the units in $\FF_2C_p$ with a view toward the study of 2-rooted primes. Several characterizations of $2$-rooted primes leading to Theorem \ref{main2} are  obtained in Section \ref{primitiveroots}. The element $1 + x + x^{2}$ in $\ftwo C_{p}$ plays an important role in this paper and is studied in Section \ref{1+x+x^2}, where we obtain a characterization of Mersenne primes in terms of binomial coefficients (see statement 7 in Theorem \ref{main}). We translate our results into the world of circulant matrices in Section \ref{matrices}. Making use of the connection between determinants, permanents, and
perfect matchings,  we provide graph theoretic characterizations of Mersenne and $2$-rooted primes in Section \ref{graphs}. In the last section we demonstrate how to tie up the various results in this paper to obtain  complete proofs of Theorems \ref{main} and \ref{main2}

\section{Group Algebras\label{section-kg}}
In this section we will prove Theorems \ref{main-2-kg} and \ref{main-p-kg}. Let $R$ be a ring and let $R^\times$ denote its group of units. If $R$ is a $\Delta_n$-ring (see \S 1 for the definition), then it is automatically a $\Delta_m$-ring whenever $n$ divides $m$.  Call $R$ a {\em strict $\Delta_n$-ring} if $n$ is the smallest positive integer such that $R$ is a $\Delta_n$-ring.  Subrings of $\Delta_n$-rings are $\Delta_n$-rings, and $R = S \times T$ is a $\Delta_n$-ring if and only if $S$ and $T$ are each $\Delta_n$-rings.

If $R$ is both a field and a $\Delta_n$-ring, we refer to $R$ as a $\Delta_n$-field. Our first lemma characterizes all $\Delta_p$-fields for $p$ a prime.

\begin{lemma}\label{delta-fields}
Let $p$ be prime.  A field $k$ is a $\Delta_p$-field if and only if $k = \mathbb{F}_2$, $k = \mathbb{F}_3$ with $p = 2$, or $k = \mathbb{F}_{p+1}$ with $p$ a Mersenne prime.
\end{lemma}
The field $\FF_2$ is in fact a $\Delta_1$-field and thus automatically a $\Delta_p$-field for all primes $p$.  The other fields are strict $\Delta_p$-fields.
\begin{proof} It is straightforward to verify the `if' direction. For the converse, suppose $k$ is a $\Delta_p$-field. Then all elements of $k^\times$ satisfy $x^p = 1$. This forces $k$ to be a finite field whose multiplicative group is necessarily cyclic, so either $|k^\times| = 1$ or $|k^\times| = p$. Hence, either $|k| = 2$ or $|k| = p+1$. If $|k| = 2$, then $k = \mathbb{F}_2$.  If $|k| = p+1$, then we obtain $\mathbb{F}_3$ if $p = 2$.  If $p > 2$, then $p+1$ must be a power of the characteristic of the field, which must be 2 since $p + 1$ is even. Thus $p$ is a Mersenne prime, and the proof is complete.
\end{proof}

We next provide a list of group algebras which are $\Delta_p$-rings.

\begin{lemma}\label{delta-examples}
Let $p$ be prime, let $k$ be a field, and let $G$ be an elementary abelian $p$-group. The group algebra $kG$ is a $\Delta_p$-ring if any of the following conditions are satisfied:
\begin{enumerate}
\item $k = \FF_3$ and $p = 2$,
\item $k = \FF_{p+1}$ and $p$ is a Mersenne prime,
\item $k = \FF_2$ and $p$ is a Mersenne prime, or
\item $k = \FF_2$ and $p = 2$.
\end{enumerate}
\end{lemma}
\begin{proof}
First, suppose $kG$ is a group algebra satisfying any of the first three conditions. Any element $t \in kG$ is a finite sum $t = \sum k_i g_i$, where $k_i \in k$ and $g_i \in G$.  Now,
\[\begin{aligned}
t^{p+1} & = \left(\sum k_i g_i\right)^{p+1} & \\
& = \sum k_i^{p+1} g_i^{p+1} & \text{(since $p+1$ is a power of $\mathrm{char}\, k$)} \\
& = \sum k_i^{p+1} g_i & \text{(since the exponent of $G$ is 1 or $p$)}\\
& = \sum k_i g_i & \text{(since $|k| - 1$ divides $p$)}\\
& = t.
\end{aligned}\]
If $t$ is a unit, we must have $t^p = 1$, so $kG$ is indeed a $\Delta_p$-ring.

Now suppose $kG = \FF_2 G$ and $p = 2$. The first step of the above argument now fails; for example, in $\FF_2 C_2 \cong \FF_2[x]/(x^2 - 1)$, we have $(1+x)^3 = 0$. Instead, we argue as follows. Consider the augmentation map $\map{\epsilon}{\FF_2G}{\FF_2}$ that sends any element to the sum of its coefficients.  Let $t = \sum k_i g_i \in kG$ be a unit.  Since $\epsilon$ is a $k$-algebra homomorphism, $\epsilon(t) = 1$. We have $$t^2 = \sum k_i^2 g_i^2 = \sum k_i^2 = \sum k_i = \epsilon(t) = 1,$$ so every unit of $\FF_2 G$ has order 2. This completes the proof.
\end{proof}

We are now able to prove Theorems \ref{main-2-kg} and \ref{main-p-kg}, which we restate for the convenience of the reader.  We state these theorems separately because although a group of exponent 2 must be abelian, a group of exponent $p > 2$ need not be abelian, so our result for $\Delta_2$-rings is stronger.

\begin{thm}\label{sec-2-kg}
Let $G$ be a group and let $k$ be a field. The group algebra $kG$ satisfies the diagonal property if and only if $kG$ is either
$\mathbb{F}_2 C_2^r $ or  $ \mathbb{F}_3 C_2^r $ for some $0 \leq r \le \infty$.
\end{thm}
\begin{proof} By Lemma \ref{delta-examples}, it suffices to prove the `only if' direction of the theorem. Let $k$ be a field and let $G$ be a group such that $kG$ is a $\Delta_2$-ring.  Since $k$ is a subring of $kG$, it is a $\Delta_2$-field.  By Lemma \ref{delta-fields}, $k = \FF_2$ or $k = \FF_3$. Since every element of $G$ is a unit, the order of every element of $G$ is a divisor of 2.  This implies $G$ is abelian, so $G$ is an elementary abelian 2-group. Since any such group is isomorphic to $C_2^r$ for some $r$ with $0 \leq r \leq \infty$, the proof is complete.
\end{proof}

\begin{thm}\label{sec-p-kg}
Let $G$ be a non-trivial abelian group, let $k$ be a field, and let $p$ be an odd prime. The group algebra $kG$ is a $\Delta_p$-ring if and only if $p$ is a Mersenne prime and
$kG$ is either  $\ftwo C_p^r$ or  $\mathbb{F}_{p+1} C_p^r$ for some $0 < r \leq \infty$.
\end{thm}
Note that the statement of the theorem requires $G$ to be non-trivial. This is simply to avoid $kG = \FF_2$, the unique $\Delta_1$-ring that is also a group algebra, because $\FF_2$ is a $\Delta_p$-ring for any prime $p$.  All $\Delta_p$-rings appearing in the statement of the theorem are strict.
\begin{proof} As in the previous proof, it suffices to prove the `only if' direction. Let $G$ be an abelian group and let $k$ be a field.  Again using Lemma \ref{delta-fields} and the fact that every element of $G$ is a unit in $kG$, we have that $G$ is a non-trivial elementary abelian $p$-group (hence $G \cong C_p^r$, where $0 < r \leq \infty$), and $k = \FF_{p+1}$ with $p$ Mersenne or $k = \FF_2$. To complete the proof, we must prove that if $\FF_2 C_p^r$ is a $\Delta_p$-ring, then $p$ is Mersenne. Since $\FF_2 C_p$ is a subring of $\FF_2 C_p^r$, we may without loss of generality assume $r = 1$. The group ring $\mathbb{F}_2 C_p$ is isomorphic to $\mathbb{F}_2[x]/(x^p - 1)$ (the isomorphism sends a generator of $C_p$ to $x$). The irreducible factors of $x^p - 1$ are distinct since this polynomial has no factors in common with its derivative $p x^{p-1}$ over $\mathbb{F}_2$ (recall that $p > 2$). By the structure theorem for modules over a principal ideal domain, this ring is therefore isomorphic to a product of fields of characteristic 2 (each factor has the form $\mathbb{F}_2[x]/(r(x))$ with $r(x)$ irreducible).  At least one such factor must have order greater than 2 since $x^p - 1$ must have at least one non-linear irreducible factor $f(x)$. The corresponding summand $\FF_2[x]/(f(x))$ must be a $\Delta_p$-field of order at least 4, so by Lemma \ref{delta-fields}, $p$ is a Mersenne prime.
\end{proof}

Although we assume $G$ is abelian in the odd primary case, one may draw a more general conclusion in one direction. For any group $G \neq \{e\}$, if $kG$ is a $\Delta_p$-ring for an odd prime $p$, then $p$ is a Mersenne prime, $k$ is a $\Delta_p$-field, $G$ has exponent $p$, and $kG$ contains $kC_p$ as a $\Delta_p$-subring. This observation, together with Theorem \ref{sec-p-kg}, demonstrates the equivalence of the first four statements of Theorem \ref{main}.
\section{Units in $\ftwo C_p$} \label{units}
With a view toward the study of 2-rooted primes in the next section, we now take a closer look at the structure of units in $\FF_2 C_p$. We begin with a useful lemma on the factorization of cyclotomic polynomials.

\begin{lemma} \label{cyclotomic}
Let $p$ be an odd prime. The cyclotomic polynomial
\[ \Phi_p(x) := 1 + x + x^2 + \cdots x^{p-1}\]
factors as a product of $(p-1)/\ord_p (2)$ distinct irreducible polynomials in $\ftwo[x]$ of degree $\ord_p (2)$, where
$\ord_p (2)$ is the smallest positive integer $t$ such that $2^{t} \equiv 1 \mod p$.
\end{lemma}

This result can be found in \cite[Page 556, Exercise 8]{DummitFoote}. We give a proof here for completeness.

\begin{proof}
Let $p(x)$ be an arbitrary irreducible factor of $\Phi_p(x)$ over $\ftwo[x]$ and let $\alpha$ be a root of $p(x)$ in $\overline{\ftwo}$.  Then we have
$\deg p(x) = \dim_{\ftwo} [\ftwo(\alpha) \colon \ftwo] $.  Since $\alpha$ is a root of $p(x)$, it is also a root of $x^p-1$, and it is not 1. Therefore  $\alpha$ is a  primitive $p$th root of $1$ in $\overline{\ftwo}$. Let $n$ be the smallest integer such that $\alpha$ is contained in $\mathbb{F}_{2^n}$.
Since the multiplicative group of a finite field is cyclic, it follows that $n$ is the smallest integer such that $\alpha^{2^n-1} = 1$ in $\overline{\ftwo}$. On the other hand  since $\alpha$ is also a primitive $p$th  root of unity, we have $\alpha^p = 1$ in  $\overline{\ftwo}$. Combining the last two facts, we conclude that $n$ is the smallest positive integer such that $2^n -1$ is a multiple of $p$.  That is, $n$ is precisely the order of $2$ in $\mathbb{F}_p$.  Thus we have
\[\deg p(x) = \dim_{\ftwo } [\ftwo(\alpha) \colon \ftwo] = \dim_{\ftwo } [\mathbb{F}_{2^n} \colon \ftwo] = n = \ord_p(2).\]
Since $p(x)$ was chosen to be an arbitrary irreducible factor, it follows that every irreducible factor of $\Phi_p(x)$ has degree $\ord_p(2)$.  Finally, note that all the irreducible factors of $x^p -1$, and hence also of $\Phi_p(x)$, are distinct because $x^p-1$ and its derivative ($px^{p-1}$) share no common factors in $\ftwo[x]$.
\end{proof}

Recall that there is a ring isomorphism
\[ \ftwo C_p \cong \frac{\ftwo[x]}{(x^p - 1)},\]
where the isomorphism takes a generator of $C_p$ to $x$.
The polynomial $x^p-1$ factors as  $(x-1)\Phi_p(x)$, where $\Phi_p(x)$ is the cyclotomic polynomial
\[ \Phi_p(x) = 1 + x+ x^2 + \cdots x^{p-1}.\]
Using Lemma \ref{cyclotomic}, we can write $x^p-1$ in $\ftwo[x]$ as
\[ x^p-1 = (x-1) \times \prod_{i = 1}^{(p-1)/\ord_p (2)} p_i(x),\]
where the polynomials $p_i(x)$ are distinct irreducible polynomials of degree $\ord_p(2)$. By the structure theorem for modules over a PID, we have
\begin{equation} \frac{\ftwo[x]}{(x^p-1)} \cong \frac{\ftwo[x]}{(x-1)}  \times \prod_{i=1}^{(p-1)/\ord_p(2)} \frac{\ftwo[x]}{(p_i(x))}.  \label{decomp}\end{equation}
Since each $p_i(x)$ is irreducible and has degree $\ord_p(2)$, we have
\[ \frac{\ftwo[x]}{(x^p-1)} \cong \ftwo \times \prod_{i=1}^{(p-1)/\ord_p(2)} \FF_{2^{\ord_p (2)}}, \]
a product of finite fields.

Taking units on both sides of the last isomorphism, we obtain
\[ \left( \frac{\ftwo[x]}{(x^p-1)} \right)^\times \cong (\ftwo)^\times \times \prod_{i=1}^{(p-1)/\ord_p(2)} \left( \FF_{2^{\ord_p (2)}} \right)^\times . \]
Since the multiplicative group of a finite field is always a cyclic group, we have the following result.

\begin{lemma}\label{numberofunits}
For any odd prime $p$,
\[ (\ftwo C_p)^\times \cong  \prod_{i=1}^{(p-1)/\ord_p(2)} C_{2^{\ord_p (2)}-1}. \]
In particular,
 \[ | (\ftwo C_p)^\times | = (2^{\ord_p (2)} -1)^{\frac{(p-1)}{\ord_p (2)}}. \]
\end{lemma}

\begin{rem}
The foregoing discussion unfolds nearly identically if one replaces the prime 2 with a prime $q \neq p$. Our decision to only consider the case $q = 2$ is motivated by our particular interest in 2-rooted primes and our desire to connect the present work to graph theory, where the entries of adjacency matrices are elements of $\FF_2$. 
\end{rem}

We now turn our attention to 2-rooted primes.

\section{$2$-rooted Primes} \label{primitiveroots}

An odd  prime $p$ is said to be $2$-rooted if $2$ is a primitive root mod  $p$. This is equivalent to saying that $2$ generates the multiplicative group of $\fp$. Therefore $p$ is $2$-rooted precisely when $\ord_p (2) = p-1$.
In this section we will characterize $2$-rooted primes by studying the units in $\ftwo C_p$. In particular, we will show that the first four statements of Theorem \ref{main2} are equivalent.

\begin{cor} \label{cor:boundforunits}
Let $p$  be an odd  prime. Then
 \[ | (\ftwo C_p)^\times | \le (2^{p-1} -1).\]
Moreover, equality holds if and only if $p$ is $2$-rooted.
\end{cor}

Note that this shows that statements (1) and (2) in Theorem \ref{main2} are equivalent.

\begin{proof}
This follows immediately from Lemma \ref{numberofunits} and the following more general and elementary fact.
If $a$ and $b$ are positive integers such that $a$ divides $b$,  then
\[ (2^a -1)^{b/a} \le 2^{b} -1,\]
and equality holds if and only if $a = b$. The corollary follows when we apply this elementary fact to $a = \ord_p(2)$ and  $b = p-1$.
\end{proof}

\begin{rem}
One can see this upper bound directly as follows. Consider the augmentation map $\epsilon \colon \ftwo C_p \rar \ftwo$ which sends an element  of $\ftwo C_p$ to the sum of its coefficients.  The kernel of this map is an $\ftwo$ subspace of index $2$ in $\ftwo C_p$. Since no element in this kernel can be a unit, we have $ |(\ftwo C_p)^\times| \le 2^{p-1}$. Moreover,  the norm element $\eta = 1 + x + x^2 + \cdots + x^{p-1}$, where is $x$ is a generator of $C_p$, is a zero divisor because $\eta (1-x) = 0$.  In particular, $\eta$ is not a unit. Since $p$ is an odd prime, $\eta$ does not belong to the kernel of $\epsilon$. Hence
$|(\ftwo C_p)^\times| \le 2^{p-1}-1$.
\end{rem}

This remark implies the following corollary which shows the equivalence of statements (1) and (4) in Theorem \ref{main2}.

\begin{cor} \label{augmentation} Let $p$ be an odd prime and let $\theta$ be an element of the ring $\ftwo C_p$.   If $\theta$ is a unit, then $\epsilon (\theta) \ne 0$ and $\theta \ne \eta$. Moreover, the converse holds precisely when  $p$ is $2$-rooted.
\end{cor}

\begin{rem}
Note that these results will allow us to reformulate  a special case of Artin's Conjecture on primitive roots stated in the introduction.
For instance, it is natural to ask whether the equality in Corollary \ref{cor:boundforunits} or the converse of the first statement in Corollary \ref{augmentation}  will hold for infinitely many primes $p$.  These  hold if and only if $2$ is a primitive root for infinitely many primes $p$, which is a special case of Artin's Conjecture.
\end{rem}

In the next proposition, we analyze  the units of order $p$ in $\mathbb{F}_2C_p$ to obtain another characterization of $2$-rooted primes.
This characterization is the equivalence of statements (1) and (3) in Theorem \ref{main2}.

\begin{prop} \label{nontrivialunits} An odd prime $p$ is $2$-rooted precisely when the only  units of order $p$ in $\mathbb{F}_2C_p$ are   the non-identity elements of $C_{p}$.
\end{prop}

\begin{proof}
Recall the structure of units in  $\mathbb{F}_2C_p$:
\[ (\ftwo C_p)^\times \cong  \prod_{i=1}^{(p-1)/\ord_p(2)} C_{2^{\ord_p (2)} -1}. \]
Since $p$ divides $2^{\ord_p(2)} - 1$, there is a unique copy of $C_p$ inside $ C_{2^{\ord_p (2)} -1}$, and thus a unique  elementary abelian $p$-group of rank $(p-1)/\ord_p(2)$ inside $(\ftwo C_p)^\times$. Therefore the number of units of order $p$ in $ \ftwo C_p$ is equal to $p^{(p-1)/\ord_p(2)} - 1$.  The only  units of order $p$ in $\mathbb{F}_2C_p$ are   the non-identity elements of $C_{p}$ if and only if
\[p^{(p-1)/\ord_p(2)} - 1 =  p  -1.\]
This equality holds if and only if $p - 1 =\ord_p(2)$, or equivalently when $2$ is  a primitive root mod $p$.
\end{proof}

What we have seen here is a striking contrast between Mersenne primes and $2$-rooted primes.  For the former primes, every non-trivial unit of $\ftwo C_p$ is of order $p$; for the latter primes, only the non-identity elements of $C_p$ will have order $p$.
This suggests that these two sets of primes are disjoint.  More precisely, the following is true.

\begin{prop} \label{3}
The prime $3$ is the only prime which is both Mersenne and $2$-rooted.
\end{prop}
\begin{proof} Let $p$ be a 2-rooted prime. Then $\ord_p(2) = p -1$ and by Lemma \ref{numberofunits}, $$(\FF_2C_p)^\times \cong C_{2^{p-1} - 1}.$$  If $p$ is also Mersenne, then $\FF_2 C_p$ is a $\Delta_p$-ring and every unit has order $p$. Since the group of units is cyclic, this forces $2^{p-1} - 1 = p$, and the only prime satisfying this equation is $p = 3$.\end{proof}

\begin{rem}
The above result is often obtained as an easy consequence of the quadratic reciprocity law. Our proof is different and it is relatively more elementary.  A noteworthy feature of our approach  is that we connect both sets of primes in question (Mersenne and $2$-rooted primes) to a common concept, namely the structure of units in $\ftwo C_{p}$.
\end{rem}

\section{The Element $1+x+x^2$} \label{1+x+x^2}

Recall that when $p$ is a Mersenne prime, every unit in 
\begin{equation}\label{short-fcp-iso}
\ftwo C_p \cong \FF_2[x]/(x^p - 1) \cong \FF_2 \times \FF_2[x]/(\Phi_p(x))
\end{equation}
has order $p$. So it is natural to ask for some explicit non-trivial examples. (The non-identity elements of the group $C_p$ are of course trivial examples of units which have order $p$.) To begin, consider the element $x^n + 1 \in \FF_2[x]/(\Phi_p(x))$ when $(n, p) = 1$.

\begin{lemma}
Let $p$ be an odd prime and let $n$ be a positive integer relatively prime to $p$. The element $x^n + 1$ is a unit in $\FF_2[x]/(\Phi_p(x))$. This unit has order $p$ if and only if $p$ is a Mersenne prime.
\end{lemma}
\begin{proof}
Let $p$ be an odd prime and let $(n, p) = 1$. If $\alpha$ is a common root of the polynomials $x^n - 1$ and $x^p - 1$ in the algebraic closure of $\FF_2$, then $\alpha^p = 1 = \alpha^n$.  This forces the multiplicative order of $\alpha$ to be 1 since $(n, p) = 1$, so $\alpha = 1$ is the only common root. Since $x^p - 1 = (x -1)\Phi_p(x)$, we obtain that $x^n - 1 = x^n + 1$ is relatively prime to $\Phi_p(x)$ in $\FF_2[x]$. Thus, $x^n + 1$ is a unit in $\FF_2[x]/(\Phi_p(x))$.

The element $x^n + 1$ has order $p$ in $\FF_2[x]/(\Phi_p(x))$ if and only if $(x+1)\Phi_p(x) = x^p - 1$ divides $$(x+1)[(x^n+1)^p - 1] = \sum_{i=1}^p \binom{p}{i} x^{ni} + \sum_{i=1}^p \binom{p}{i} x^{ni+1}.$$ Working modulo $x^p - 1$, we may reduce powers of $x$ modulo $p$ and the resulting coefficients of $1, x, \dots, x^{p-1}$ must all be zero modulo 2. We therefore obtain, for each degree $0 \leq k \leq p-1$, $$ \binom{p}{n^{-1}k \mod p} \equiv \binom{p}{n^{-1}k - n^{-1} \mod p}\mod 2.$$ Taken together, these congruences imply that the binomial coefficients $\binom{p}{0}, \dots, \binom{p}{p-1}$ are mutually congruent modulo 2 and therefore all congruent to $\binom{p}{0} = 1$. This last condition is equivalent to $p$ being a Mersenne prime (see \cite[Theorem 8.14]{Kosher}).
\end{proof} 
The element $x^n + 1$ in the lemma above lifts to the unit $1 + x^n + \Phi_p(x)$ in $\FF_2 C_p$ (use the isomorphism in equation (\ref{short-fcp-iso}) at the beginning of this section to see this), providing an example of a single unit whose order is $p$ if and only if $p$ is a Mersenne prime. In the remainder of this section, we examine another such example, $1 + x + x^2$. It is noteworthy that it is an irreducible polynomial whose degree does not depend on $p$.

\begin{thm}  \label{15}  Let $p > 3$ be a prime and let $x$ be a  generator of the cyclic group $C_p$. Then,
$(1 +  x + x^2)^p = 1$ in $\mathbb{F}_2C_p$  if and only if  $p$ is a Mersenne prime.
\end{thm}

\noindent Though there are other units in $\ftwo C_p$ which will have order $p$ precisely when $p$ is Mersenne,  the unit $1 + x + x^2$ is the `smallest' example with this property, because $1 + x$ is not a unit and $x$ has order $p$ for {\em any} prime. This result  is exactly the equivalence of statements (1) and (5) of Theorem \ref{main}.  In this section, we will prove  this equivalence by showing that
\[ (1) \implies (5) \implies (6) \implies (7) \implies (1)\]
in Theorem \ref{main}. 

Since every unit in $\ftwo C_p$ will have order $p$ when $p$ is Mersenne, to establish the `if' part  of Theorem \ref{15}, it is enough to show that $1 + x + x^2$ is a unit in $\FF_2[x]/(x^p - 1)$ when $p > 3$. This is indeed the case: when $p > 3$, $\ord_p(2) > 2$, so the degree of every irreducible factor of $\Phi_p(x)$ is greater than 2 by Lemma \ref{cyclotomic}. Hence, $1 + x + x^2$ and $\Phi_p(x)$ are relatively prime, and $1 + x + x^2$ is a unit in $\FF_2 C_p$.

In remainder of this section, we will prove the converse. That is, we will show that  if $p > 3$ is a prime and $(1+ x + x^2)^p = 1$ in $\ftwo C_p$, then $p$ is Mersenne. We will do this by showing $ (5) \implies (6) \implies (7) \implies (1)$ in Theorem \ref{main}.

To this end, we need a formula of Lucas which gives an efficient algorithm for computing the binomial coefficients mod $2$, and a characterization of Mersenne primes in terms of binomial coefficients.

\begin{thm}[Lucas] Let $m$ and $n$ be positive integers.  Then the binomial coefficients mod $2$ can be computed using the formula:
\[ {m \choose n} = \prod_{i=0}^k { m_i \choose n_i} \mod 2\]
where
\[  m = \sum_{i=0}^k m_i 2^i \ \ \ \text{and  } \ \ \  n = \sum_{i=0}^k n_i 2^i \]
are the expansions of the integers $m$ and $n$ respectively  in base $2$.\label{Lucas}
\end{thm}

The next proposition can be easily deduced from Lucas' theorem.

\begin{prop} \cite{Kosher} \label{2^i} Let $p$ be an odd prime. Then $p$ is Mersenne if and only if ${p\choose 2^m}=1\mod 2$ for all $m$ such that $0 \le 2^m \le p$.
\end{prop}

The next proposition shows $ (5) \implies (6) \implies (7) $ in Theorem \ref{main}.

\begin{prop} \label{56} Let $p>3$ be an odd prime and  let $x$ be a generator for the cyclic group $C_p$. If $(1+x+x^2)^p=1$ in $\ftwo C_p$, then
\[ {p\choose j}\equiv{p\choose 3j \mod p }\mod 2 \ \  \text{for }  0 \le j \le p. \]
\end{prop}
\begin{proof} Since $(1+x+x^2)(1+x)=1+x^3$ in  $\ftwo C_p$, raising to the $p$th powers on both sides, we get  $$(1+x+x^2)^p(1+x)^p=(1+x^3)^p.$$
Since  $(1+x+x^2)^p=1$,  we have
$$(1+x)^p=(1+x^3)^p.$$
Expanding both these expression using the binomial series, we get
$$\sum_{i=0}^{p}{p\choose i}x^i=\sum_{j=0}^{p}{p\choose j}x^{3j}.$$
This  last equation holds in the group algebra $\ftwo C_p$ where we can equate the coefficients of like powers of $x$.  This gives the desired result:
${p\choose j}\equiv{p\choose 3j\mod p}\mod 2$ for all $0\leq j\leq p$.
\end{proof}

We will now show that  the condition
\[ {p\choose j}\equiv{p\choose 3j \mod p }\mod 2 \ \  \text{for }  0 \le j \le p. \]
implies that $p$ is Mersenne.  To this end, we need the following lemma.

\begin{lemma} \label{2facts} Let $p$ be an odd prime and let $k$ be a nonnegative integer. Then we have the following.
\begin{enumerate}
\item ${p\choose 2k}\equiv {p\choose 2k+1}\mod 2$  for all $k$.
\item  If ${p\choose 4k+2}\equiv 1\mod 2$ then ${p\choose 4k}\equiv 1\mod 2$.
\end{enumerate}
\end{lemma}
\begin{proof} One can easily verify that  $(2k+1){p\choose 2k+1}={p\choose 2k}(p-2k)$.  Now, since $p-2k$ and  $2k+1$ are both odd, it follows that
${p\choose 2k+1}\equiv {p\choose 2k}\mod 2$.

For  the second statement, we first verify that $$(4k+1)(2k+1){p\choose 4k+2}={p\choose 4k}\frac{(p-4k)(p-4k-1)}{2}.$$ Now we note that   $(4k+1)(2k+1)$ is odd and $(p-4k)(p-4k-1)$ is even. So we have
${p\choose 4k+2}\equiv {p\choose 4k}\frac{(p-4k)(p-4k-1)}{2}\mod 2$.
 It is now clear that if ${p\choose 4k}$ is even, then so is ${p\choose 4k+2}$.
\end{proof}

\begin{prop} \label{71} Let $p > 3$ be a prime such that ${p\choose j}\equiv{p\choose {3j\mod p}}\mod2$ for all $1  \le j \le  p-1$. Then $p$ is Mersenne.
\end{prop}
\begin{proof} By Proposition \ref{2^i}, it is enough to show that ${p\choose 2^i}\equiv1\mod 2$ for all $i$ such that $0 \le  2^{i} \le p $. We will prove this  using induction on $i$.
When $i = 0$,  ${p \choose 2^0} = {p\choose 1} =  p \equiv 1 \mod 2$ because $p$ is odd. When $i = 1$,  ${p \choose 2^1} \equiv {p\choose 2} \equiv {p \choose 3} \equiv {p \choose 1} \equiv p \equiv 1 \mod 2$.  (Here the second congruence follows from part 1 of Lemma \ref{2facts} and the third follows from the given hypothesis.)
For the induction step, we assume that ${p\choose 2^i}\equiv 1\mod 2$ for all $i<s$, and we will show that ${p\choose 2^s}\equiv 1\mod 2$.
Let $p=\sum_{i=0}^{n} a_i 2^i$ be the base $2$ expansions of  $p$. We first claim that all the $a_i$'s for $i < s$ have to be 1. This is so because for $i$ in this range, using the induction hypothesis, we have, $1 \equiv {p \choose 2^i} \equiv {a_i \choose 1} = a_i \mod 2$.
Since $2^s$ is not a multiple of 3, note that either  $2^s+1$ or $2^s+2$ has to be a multiple of 3.
\vskip 2mm
\noindent
\textbf{Case 1:}  Suppose that  $2^s+1$ is a multiple of 3.   Then we can write $2^s+1=3k$ for some $k$ where $k<2^s$. We let $k=\sum_{i=0}^{s-1}\alpha_i2^i$ be the base $2$ expansions of $k$.

Note that by part 1 of Lemma \ref{2facts}, it is enough to show that  ${p\choose 2^s+1} \equiv 1 \mod 2$.  Using the given hypothesis, applying Lucas theorem, and the fact that $a_i = 1$ for $i < s$, we get (working mod 2)
$${p\choose 2^s+1}={p\choose 3k}\equiv{p\choose k}={ \sum_{i=0}^{n} a_i 2^i \choose {\sum_{i=0}^{s-1}\a_i2^i}}=\prod_{i = 1}^{s-1} {a_i\choose \alpha_i}\equiv   \prod {1 \choose \alpha_i} \equiv \prod(1)= 1.$$ Hence, ${p\choose 2^s} \equiv 1 \mod 2$.

\vskip 2mm
\noindent
\textbf{Case 2:}  Suppose $2^s+2$ is a multiple of 3.  This case is similar to Case 1 except that we now use part 2 of Lemma \ref{2facts}.
Let  $2^s+2=3r$ for some $r$ where $r<2^s$, and let  $r=\sum_{i=0}^{s-1}\beta_i2^i$ be the base 2 expansion of $r$.  As before, working mod 2, we have
\[ {p\choose 2^s+2}={p\choose 3r}\equiv {p\choose r}={  \sum_{i=0}^{n} a_i 2^i \choose\sum_{i=0}^{s-1}\beta_i2^i} \equiv \prod_{i = 1}^{s-1} {a_i\choose \beta_i}\equiv   \prod {1 \choose \beta_i} \equiv \prod(1)=1. \]
Consequently, ${p\choose 2^s}\equiv 1\mod 2$.
\end{proof}

This completes the proof of the main theorem of this section.

\section{Circulant Matrices} \label{matrices}

An $n \times n$ square matrix $C$ over a field $k$ is circulant if each column of $C$ is obtained by rotating one element down relative to its preceding column.  Thus, a circulant matrix  is completely determined by specifying the first column ($\bf{v}$) because all the remaining column vectors are each cyclic permutations of $\bf{v}$ with offset equal to the column index. We denote this by $\Circ(\bf{v})$. The collection of all $n \times n$ circulant matrices over a field $k$ forms a ring, and we  denoted by $Cr_n(k)$.

It turns out that the results in this paper can be formulated in terms of circulant matrices. To see this connection, let $x$ be a generator of $C_n$, and  consider the natural map
\[ \rho \colon k C_n \lrar Cr_n(k)\]
defined by $\rho(\Sigma_{i=0}^{n-1} \alpha_i x^i) = \Circ(\alpha_0, \alpha_1, \alpha_2, \cdots, \alpha_{n-1})$.

\begin{prop}
The map $\rho$ establishes an isomorphism between the rings $kC_n$ and $Cr_n(k)$.
\end{prop}

The proof is a straightforward verification, and we leave it to the reader. Under this isomorphism, we have the following dictionary.

\begin{enumerate}
\item Units in $kC_n$ correspond to the invertible matrices.
\item Group elements correspond to circulant permutation  matrices.
\item  The identity element of $k C_n$ corresponds to the identity matrix.
\item The elements $1+x$, $1+x^3$, $1 + x + x^2$ in $\ftwo C_n$ correspond respectively to the circulant matrices $\Circ(1,1,0, 0, \cdots, 0)$, $\Circ(1,0,0,1, 0, \cdots 0)$, and $\Circ(1, 1, 1, 0, 0, \cdots, 0)$.
\item The norm element $1+x+x^2+ \cdots x^{n-1}$ in $k C_n$ corresponds to the $n \times n$ matrix $J$ which consists of all 1's.
\end{enumerate}

With this dictionary at hand, we can immediately translate our results to the world of circulant matrices.

\begin{thm} \label{8910}
Let $p > 3$ be a prime. Then the following are equivalent.
\begin{enumerate}
\item $p$ is a Mersenne prime.
\item[(8)] The group of invertible $p \times p$ circulant matrices over $\ftwo$ is an elementary abelian $p$-group.
\item[(9)]  $[\Circ(1, 1, 1, 0, 0, \cdots, 0)]^p = I_p \mod 2$.
\item[(10)] $[\Circ(1,1,0, 0, \cdots, 0)]^p = [\Circ(1,0,0,1, 0, \cdots 0)]^p \mod 2$.
\end{enumerate}
\end{thm}
\begin{proof}
We consider the isomorphism
\[ \rho \colon \ftwo C_p \lrar Cr_p(\ftwo)\]
defined above.  Under this isomorphism (which gives the above dictionary),  statements (8), (9), and (10) are equivalent respectively to statements (4), (5) and (6) of Theorem \ref{main}.  The latter statements were already shown to be equivalent to statement (1)  in Section \ref{units} and Section \ref{1+x+x^2}. So we are done.
\end{proof}

 We now translate the characterizations of $2$-rooted primes in the language of circulant matrices.

\begin{thm} \label{567} Let $p$ be an odd prime. Then the following are equivalent.
\begin{enumerate}
\item $p$ is $2$-rooted.
\item[(5)] There are $2^{p-1}-1$ invertible circulant matrices of size $p \times p$ over $\ftwo$.
\item[(6)] The only invertible circulant matrices over $\ftwo$ that have order dividing $p$ are the circulant permutation matrices.
\item[(7)] If $A$ is a  $p \times p$ circulant matrix over $\ftwo$ which is not $J$ (matrix with all 1's) and the vector of all 1's is not in the null space of $A$, then $A$ is invertible over $\ftwo$.
\end{enumerate}
\end{thm}

\begin{proof}
Once again we use the aforementioned dictionary given by  the isomorphism
\[ \rho \colon \ftwo C_p \lrar Cr_p(\ftwo).\]
Under this isomorphism, statements (5), (6) and (7) are equivalent respectively to statements (2), (3) and (4) of Theorem \ref{main2}.  To see the equivalence of (4) and (7), observe that an element $\sum_{i=0}^{p-1} a_i x^i$ in $\ftwo C_p$ will not be in the kernel of the augmentation map exactly when  $\sum a_i$ is equal to $1$.
This is equivalent to saying that the $p \times 1$ vector of all $1$'s is not in the null space of the circulant matrix $\Circ(a_0, a_1, \cdots, a_{p-1})$.
Statements (2), (3) and (4) of Theorem \ref{main2} were already shown to be equivalent to (1) in Section \ref{primitiveroots}. So we are done.
\end{proof}

\section{Bipartite Graphs} \label{graphs}

In this section we will make graph theoretic translations of our results using simple definitions and ideas from graph theory. We use the standard terminology of graphs which can be found in any textbook on graph theory; see \cite{Bollobas} for instance.

Given any $n \times n$ binary matrix (one in which every entry is either a $0$ or a $1$) $M = (m_{ij})$, we can associate to it an $(n , n)$ bipartite graph as follows. Take two sets $A := \{ a_1, a_2, \cdots, a_n\}$ and $B := \{ b_1, b_2, \cdots, b_n\}$. Vertex $a_i$ is adjacent to $b_j$ if and only if $m_{ij} = 1$.  This association clearly establishes a 1-1 correspondence between $n \times n$ binary matrices and the collection of bipartite graphs on sets $A$ and $B$.  The matrix corresponding to a graph in this bijection is often called the biadjacency matrix of the graph.
A bipartite graph $G$ is called  circulant if its biadjacency matrix is a circulant matrix.

Some algebraic invariants and operations in the world of matrices can be interpreted graph theoretically.  Here we will discuss  the graph theoretic  interpretation of the determinant and matrix multiplication. To do this we need a few definitions.

The permanent of an $n \times n$ square matrix $T =(t_{ij})$ is defined as
\[ \text{perm}(T) = \sum_{\pi \in S_n} t_{1\pi(1)}t_{2\pi(2)}\cdots t_{n \pi(n)},\]
where $S_n$ is the set of all permutations of  the set $\{ 1, 2, 3, \cdots , n\}$.
This looks almost like the definition of the determinant. The only difference is that we do not have the extra $\text{sgn}(\pi)$ in front of each term in the above sum.  Therefore, note that when working modulo $2$, the two notions are the same. That is,
\[ \det(T) \equiv \text{perm}(T)  \mod 2.\]

A matching in a graph $G$ is a set of edges $F \subseteq E(G)$ such that no vertex of $G$ is incident to more than one edge of $F$. A perfect matching is a matching that will cover all the vertices of $G$. Matching theory is a rich branch of graph theory and we refer the reader to the excellent book \cite{lovasz} for a wealth of useful information on matchings.

When $M$ is the biadjacency matrix of an $(n, n)$-bipartite graph as explained above, then $\text{perm}(M)$ is equal to the number of perfect matchings in $G$. (This is easy to see. Note that every perfect matching corresponds to a permutation $\pi$ in $S_n$ such that $m_{i\pi(i)} = 1$ for all $i$.
Therefore  in the formula for the permanent,
\[ \text{perm}(M) = \sum_{\pi \in S_n} m_{1\pi(1)}m_{2\pi(2)}\cdots m_{n \pi(n)},\]
a term will be equal to 1 precisely when $\pi$ corresponds to a perfect matching, and will be $0$ otherwise.) This gives:

\begin{lemma}
Let $G$ be an $(n, n)$ bipartite graph. $G$ has an odd number of perfect matchings if and only if the biadjacency matrix $M$ of $G$ is invertible mod $2$.
\end{lemma}

To explain matrix multiplication for biadjacency matrices  graph theoretically, we need one more definition.
A pseudopath of length $r$ in an $(n, n)$  bipartite graph $G$ is an ordered sequence $\{e_1, e_2, \cdots, e_r \}$ of $r$ edges in $G$  such that for all $i$ from 1 to $r-1$, the tail of $e_i$ in $B$ and the head of $e_{i+1}$  in $A$ have the same subscript.
One can now easily verify that the $(i, j)$th entry in $M^r$ counts the number $s_{ij}(r)$ of pseudo paths of length $r$ in $G$ between $a_i$ and $b_j$.

Let $\G_p$ be the collection of all $(p,p)$ labeled bipartite circulant graphs. It is easy to see that these graphs are regular (i.e, all vertices have the same degree).
We now have a natural bijection from $\G_p$ to the ring of $p \times p$ circulant matrices over $\ftwo$.
\[ \eta  \colon \G_p \lrar Cr_p(\ftwo),\]
which assigns to each graph in $\G_p$ its biadjacency matrix.  Note that this is a well-defined map, which endows $\G_p$ the structure of an unital associative ring.  This isomorphism of rings gives the following dictionary in view of the above discussion.

\begin{enumerate}
\item Graphs which have an odd number of perfect matchings correspond to invertible matrices.
\item The identity matrix corresponds to the graph of the trivial perfect matching in which $a_i$ is adjacent to  $b_i$ for all $i$.
\item Non-identity matrices of order $p$ correspond to graphs which have the property that  $s_{ij}(p)$ (the number of pseudopaths between vertex $a_i$ and vertex $b_j$) is equal to $\delta_{ij}$, the Kronecker delta symbol.
\item The matrix $J$ corresponds to the complete bipartite graph.
\item The degree of the graph will be the sum of the entries in any row or column of the biadjacency matrix.
\item Graphs of degree 1 correspond to circulant permutation matrices.
\end{enumerate}

We are now ready to translate our algebraic results into the world of circulant bipartite graphs.

\begin{prop}
There are exactly $(2^{\ord_p(2)} - 1)^{(p-1)/\ord_p(2)}$,  labeled $(p, p)$  bipartite circulant graphs which have an odd number of perfect matchings. Moreover, this collection is naturally equipped with a structure of an abelian group.
\end{prop}

\begin{proof}Under the  isomorphisms $\eta  \colon \G_p \lrar Cr_p(\ftwo)$ and $\rho \colon \ftwo C_p \lrar Cr_p(\ftwo)$,
the collection in question is exactly equal to $(\ftwo C_p)^\times$. This is the abelian group of units in $\ftwo C_p$ and its structure and order was computed in Section \ref{units}.
\end{proof}

\begin{prop}
Let $p > 3$ be a prime. The $(p, p)$ bipartite graph corresponding to the $p \times p$ Circulant matrix $\Circ(1, 1,1, 0 \cdots , 0)$
will have an odd number of perfect matchings.
\end{prop}

\begin{proof}
It is enough to show that the $p \times p$ matrix  $\Circ(1, 1,1, 0 \cdots , 0)$  is invertible over $\ftwo$ whenever $p > 3$. This is equivalent to showing that the element $1 + x +x^2$ is a unit in $\ftwo C_p$, where $x$ is a generator of $C_p$. We proved this in  Section \ref{1+x+x^2}.
\end{proof}

\begin{thm} \label{1112}Let $p > 3$ be a prime. Then the following are equivalent.
\begin{enumerate}
\item $p$ is a Mersenne
\item[(11)] Every circulant  $(p ,p)$ bipartite graph with odd number of perfect matchings has
$s_{ij}(p) \mod 2 = \delta_{ij}$, where $s_{ij}(p)$ is the number of pseudopaths between vertex $a_i$ and vertex $b_j$, and $\delta_{ij}$ is the Kronecker delta symbol.
\item[(12)] The $(p,p)$ bipartite graph corresponding to the $p \times p$ circulant matrix  \break $\Circ(1, 1, 1, 0, 0, \cdots, 0)$ has
$s_{ij}(p) \mod 2 = \delta_{ij}$.
\end{enumerate}
\end{thm}

\begin{proof} Using the dictionary given by the map $\eta$, we see that statements (11) and (12) are equivalent respectively  to statements (8) and (9) of Theorem \ref{main}. The latter were shown to be equivalent to (1) in Section \ref{matrices}.
\end{proof}

\begin{thm} \label{8910m2} Let $p$ be an odd prime. Then the following are equivalent.
\begin{enumerate}
\item $p$ is $2$-rooted.
\item[(8)] There are $2^{p-1}-1$  circulant $(p,p)$ bipartite graphs on labeled vertices with odd number of perfect matchings.
\item[(9)] If $G$ is a circulant $(p,p)$ bipartite graph on labeled vertices with an odd degree and is not a complete bipartite graph, then $G$ has an odd number of perfect matchings.
\item[(10)] If $G$ is a circulant $(p,p)$ bipartite graph on labeled vertices with an odd number of perfect matchings and $s_{ij}(p) \mod 2  = \delta_{ij}$ for all $1\le i , j \le p$, then the degree of $G$ is 1.
\end{enumerate}
\end{thm}

\begin{proof} Using the  dictionary given by the map $\eta$ we see that statements (8), (9) and (10) are equivalent respectively to statements (5), (7) and (6) of Theorem \ref{main2}. The latter were shown to be equivalent to (1) in Section \ref{matrices}.
\end{proof}

\section{Proofs of Theorems \ref{main} and \ref{main2}}

In this section we will explain how to tie up the various results in this paper to complete the proofs of Theorem \ref{main} and Theorem \ref{main2}.
\begin{proof}[Proof of Theorem \ref{main}:] It was observed at the end of \S 2 that statement (3) is equivalent to statement (2). The following diagram shows how to justify the remaining equivalences in Theorem \ref{main}.
\begin{align*}
(1) \overset{\ref{sec-p-kg}}{\iff}& (2)                                                                                         \\
(1) \overset{\ref{sec-p-kg}}{\iff}     & (4)  \overset{\ref{8910}}{\iff}  (8)  \overset{\ref{1112}}{\iff}   (11)   \\
                                               & (5) \overset{\ref{8910}}{\iff}  (9)  \overset{\ref{1112}}{\iff}    (12)    \\
                                               & (6)  \overset{\ref{8910}}{\iff} (10)                                           \\
(1) \overset{\ref{15}}{\implies}           & (5) \overset{\ref{56}}{\implies}    (6) \overset{\ref{56}}{\implies}       (7) \overset{\ref{71}}{\implies}  (1)
\end{align*}

\end{proof}

\begin{proof}[Proof of Theorem \ref{main2}:] The following diagram shows how the 10 statements of Theorem \ref{main2} are equivalent.
\begin{eqnarray*}
(1) \overset{\ref{cor:boundforunits}}{\iff} & (2) \overset{\ref{567}}{\iff}  &  (5) \overset{\ref{8910m2}}{\iff} (8) \\
(1) \overset{\ref{nontrivialunits}}{\iff} & (3) \overset{ \ref{567}}{\iff}  &  (6) \overset{\ref{8910m2}}{\iff} (10) \\
(1) \overset{ \ref{augmentation}}{\iff} & (4) \overset{ \ref{567}}{\iff}  &  (7) \overset{ \ref{8910m2}}{\iff} (9)
\end{eqnarray*}

It remains to show that each of these statements  imply that $p \equiv 3 \text{ or } 5 \mod 8$. Since these statements are all equivalent, it is enough to show that statement (1) implies this condition on $p$.  To this end, let $p$ be a $2$-rooted prime. It is easy to see that $2$ is a quadratic non-residue mod $p$;  that is, $2$ is not a square mod $p$. (For, if $2 \equiv u^{2} \mod p$, then $2^{\frac{p-1}{2}} = 1$ by Fermat's little theorem, contradicting the fact that $p$ is $2$-rooted.) Using the Legendre symbol, this can be expressed as
\[ \left(\frac{2}{p}\right) = -1.\]
From the quadratic reciprocity law, we know that this equation holds precisely when $p \equiv 3 \text{ or } 5 \mod 8$; see \cite{Burton}.
\end{proof}

\vskip 3mm
\section*{Acknowledgements} The first author presented this research in the number theory seminar at UIUC and in the Discrete Mathematics seminar at Illinois State University.  We would like to thank the valuable feedback we received from these groups. In particular, we are thankful to Papa Sissokho for pointing out to us the connection between the permanent of a matrix and perfect matchings of a graph. We would also like to thank Keir  Pieter Moore, Jan Minac,  B. Sury, and Shailesh Tipnis  for  their interest, sharing their thoughts, and giving us valuable references related to this work. This research was inspired by a graduate course in number theory  taught by the first author in the fall of 2013. We would like to thank our graduate student Christina Henry for  her help in improving the exposition.

\end{document}